\documentclass[10pt]{article}
\usepackage[T1]{fontenc}
\usepackage[cp1250]{inputenc}
\usepackage{lmodern}

\usepackage[english]{babel}
\usepackage{latexsym}
\usepackage{amsmath}
\usepackage{indentfirst}
\usepackage{amsthm}
\usepackage{amssymb}
\usepackage{amsfonts}
\usepackage{stackrel}
\usepackage{dsfont}
\usepackage{tikz}
\usepackage{slashbox}
\usepackage[mathscr]{eucal}
\usepackage{authblk}
\usepackage{hyperref}
\usepackage{mathabx}
\usepackage{bookmark}

\newtheorem{thm}{Theorem}

\newcommand{\reals}{\mathbb{R}}
\newcommand{\naturals}{\mathbb{N}}
\newcommand{\integers}{\mathbb{Z}}
\newcommand{\complex}{\mathbb{C}}
\newcommand{\eps}{\varepsilon}

\newcommand{\Fourier}{\mathcal{F}}
\newcommand{\Hankel}{\mathcal{H}}
\newcommand{\Hilbert}{\mathcal{H}}

\begin{document}
\title{Characterizing the Fourier transform by its properties}
\author{Mateusz Krukowski}
\affil{Institute of Mathematics, \L\'od\'z University of Technology, \\ W\'ol\-cza\'n\-ska 215, \
90-924 \ \L\'od\'z, \ Poland \\ \vspace{0.3cm} e-mail: mateusz.krukowski@p.lodz.pl}
\maketitle

\begin{abstract}
It is common knowledge that the Fourier transform enjoys the convolution property, i.e., it turns convolution in the time domain into multiplication in the frequency domain. It is probably less known that this property characterizes the Fourier transform amongst all linear and bounded operators $T:L^1 \longrightarrow C^b.$ Thus, a natural question arises: are there other features characterizing Fourier transform besides the convolution property? We answer this query in the affirmative by investigating the time differentiation property and its discrete counterpart, used to characterize discrete-time Fourier transform. Next, we move on
to locally compact abelian groups, where differentiation becomes meaningless, but the Fourier transform can be characterized via time shifts. The penultimate section of the paper returns to the convolution characterization, this time in the context of compact (not necessarily abelian) groups. We demonstrate that the proof existing in the literature can be greatly simplified. Lastly, we hint at the possibility of other transforms being characterized by their properties and demonstrate that the Hankel transform may be characterized by a Bessel-type differential property.
\end{abstract}

\vspace{0.2cm}

\section{Introduction}
\label{Chapter:Introduction}

Fourier transform is one of the central topics in harmonic analysis. Its significance is boosted by numerous applications in fields such as the analysis of differential equations, nuclear magnetic resonance, infrared spectroscopy, signal/image processing or quantum mechanics. Amongst many attempts at explaining why the Fourier transform is so exceptional, there is a group of arguments which follow a similar pattern. They boil down to the reasoning that the Fourier transform is the only operator satisfying a specified set of properties. One may regard this approach as ``axiomatic'' in the sense that we do not start with a complete formula for the Fourier transform, but rather demand certain properties of an operator and derive the formula as a necessary conclusion.

To give an example of such an approach, in his paper ``A characterization of Fourier transforms'' (see \cite{Jaming}) Jamming proved that the convolution property 
$$\forall_{f,g\in L^1}\ \Fourier(f\star g) = \Fourier(f) \Fourier(g)$$

\noindent
essentially characterizes the Fourier transform on real numbers $\reals$, the circle group $S^1$, the integers $\integers$ and the finite cyclic group $\integers_n$. His article inspired Lavanya and Thangavelu to show that any continuous *-homomorphism of $L^1(\complex^d)$ (with twisted convolution as multiplication) into $B(L^2(\reals^d))$ is essentially a Weyl transform and deduce a similar characterization for the Fourier transform on the Heisenberg group (see \cite{LavanyaThangavelu} and \cite{LavanyaThangavelu2}). Furthermore, Kumar and Sivananthan went on to demonstrate that the convolution property characterizes the Fourier transform on all compact groups (see \cite{KumarSivananthan}). 

It is easy to recognize a common thread in the articles cited above -- they all focus on the convolution property as the key ingredient in characterizing the Fourier transform. It is thus quite natural to ponder the question: are there any other properties that distinguish the Fourier transform? Our goal is to answer this question with a resounding ``yes!''

The paper comprises of the introduction, followed by four thematic sections and a bibliography. Section \ref{Section:Fourieronreals} provides a brief summary of the classical Fourier transform, with an emphasis on its time differentiation property. We extend this property to a larger class of integral transforms and introduce the Dirac delta property, which is a kind of ``boundary condition''. Theorem \ref{characterizationofFouriertransform} is the climactic point of the section, where we demonstrate that the time differentiation property (coupled with Dirac delta property) characterizes the classical Fourier transform on $\reals$. We go on to prove that a similar technique (suitably adjusted to $\integers$ instead of $\reals$) can be used to characterize the discrete-time Fourier transform.

Section \ref{Section:LCA} is preoccupied with the Fourier transform on an arbitrary locally compact abelian group. Needless to say, it makes little sense to talk of differentiation in this context so we study the time shift property instead. We demonstrate that this condition characterizes the Fourier transform with the help of (suitably reformulated) Dirac delta property (see Theorem \ref{Fouriershiftcharacterization}).

Section \ref{Section:compactgroups} commences with a recap of representation theory on compact groups. We work towards describing a bijection between irreducible, unitary representations of $G$ and irreducible, unitary *-representations of $L^1(G)$. This correspondence becomes the fundamental tool in demonstrating that the convolution property characterizes the Fourier transform on compact groups. Arguably, our proof simplifies earlier works of Kumar and Sivananthan (see \cite{KumarSivananthan}).

Section \ref{Section:Hankel} is meant to stimulate further research regarding other integral transforms. We prove that the Fourier transform is not the only operator which is characterized by some differentiation property. After a brief summary of the Hankel transform, Theorem \ref{Hankelcharacterization} demonstrates that the Bessel's differentiation property (together with a suitable ``boundary condition'') characterizes the Hankel transform. The paper concludes with a bibliography.

\section{Fourier transform on reals and integers}
\label{Section:Fourieronreals}

It is common knowledge that the Fourier transform 
$$\Fourier(f) := \int_{\reals}\ f(x)e^{-2\pi ixy}\ dx$$

\noindent
satisfies the equality 
\begin{gather}
\Fourier(f')(y) = -2\pi iy \Fourier(f)(y)
\label{Fourierdifferentiation}
\end{gather}

\noindent 
for every $y\in\reals$ and sufficiently ``good'' function $f$. The property is so elemental that it permeates to the realm of quantum mechanics, where it states that the Fourier transform of the momentum operator 
$$\hat{p} := -i\hbar \frac{\partial}{\partial x}$$ 

\noindent
is (up to a constant) the position operator 
$$\hat{x}\psi(x,t) := x\psi(x,t).$$ 

\noindent
Our aim in the current section is to argue that the time differentiation property \eqref{Fourierdifferentiation} is so fundamental that it characterizes the Fourier transform (given an appropriate ``boundary condition''). 

To begin with, let
\begin{itemize}
	\item $C^b(\reals)$ be the Banach space of complex-valued, continuous and bounded functions on $\reals,$
	\item $C_c^1(\reals)$ be the vector space of complex-valued, compactly-supported and $C^1-$functions on $\reals,$
	\item $L^1(\reals)$ be the Banach space of complex-valued, integrable functions on $\reals$.
\end{itemize}

\noindent
We focus on integral transforms $T: L^1(\reals)\longrightarrow C^b(\reals)$, i.e., 
$$T(f)(y) := \int_{\reals}\ K(x,y)f(x)\ dx,$$

\noindent
where $K\in C^b(\reals\times\reals)$ is differentiable with respect to the first variable. Amongst all these operators we distinguish those that satisfy the time differentiation property, i.e.,
\begin{gather}
T(f')(y) = -2\pi iy T(f)(y)
\label{differentiationproperty}
\end{gather}

\noindent
for every $f\in C_c^1(\reals)$ and $y\in\reals$. The condition is plainly molded in the image of \eqref{Fourierdifferentiation}, although we cannot expect it to characterize the Fourier transform alone. This is because the zero map is a perfect example of an integral transform satisfying \eqref{differentiationproperty}, which is not the Fourier transform. Thus, we need some sort of ``boundary condition'' accompanying the time differentiation property. 

In order to introduce such a condition, we say that $(\delta_n) \subset C_c^1(\reals)$ is a Dirac delta sequence if every $\delta_n$
\begin{itemize}
	\item has compact support contained in $\left[-\frac{1}{n},\frac{1}{n}\right],$
	\item is nonnegative and symmetric, i.e. $\delta_n(x) = \delta_n(-x)$ for every $x \in\reals,$
	\item satisfies the equality
	$$\int_{\reals}\ \delta_n(x)\ dx = 1.$$
\end{itemize}

\noindent
An example of a Dirac delta sequence is 
\begin{gather*}
\delta_n(x) := 
\left\{\begin{array}{cc}
C n e^{-\frac{1}{1-(nx)^2}} & \text{ if } x\in \left(-\frac{1}{n},\frac{1}{n}\right),\\
0 & \text{ otherwise,}
\end{array}\right.
\end{gather*}

\noindent
where
$$C := \left(\int_{-1}^1\ e^{-\frac{1}{1-x^2}}\ dx\right)^{-1}.$$ 

\noindent
The choice of constant $C$ is such that 
$$\int_{\reals}\ \delta_n(x)\ dx = 1$$

\noindent
for every $n\in\naturals.$ 

We are ready to introduce the ``boundary condition'' accompanying the time differentiation property. We say that $T: L^1(\reals) \longrightarrow C^b(\reals)$ satisfies the Dirac delta property if there exists a Dirac delta sequence $(\delta_n) \subset C_c^1(\reals)$ such that for every $y\in\reals$ we have 
\begin{gather}
\lim_{n \rightarrow \infty}\ T(\delta_n)(y) = 1.
\label{Diracdeltapropertyonreals}
\end{gather}

\noindent 
Intuitively, the property means that $T$ maps (certain) functions ``compressed in time'' to functions ``spread in frequency''. It is not surprising that the Fourier transform has the Dirac delta property (which is a good indicator that we have not strayed from the set path). Indeed, for every $y\in\reals$ and $n\in\naturals$ we have 
\begin{equation*}
\begin{split}
\Fourier(\delta_n)(y) &= \int_{\reals}\ \delta_n(x)e^{-2\pi ixy}\ dx = \int_{(-\infty,0)}\ \delta_n(x)e^{-2\pi ixy}\ dx + \int_{(0,\infty)}\ \delta_n(x)e^{-2\pi ixy}\ dx\\
&=\int_{(0,\infty)}\ \delta_n(-x)e^{2\pi ixy}\ dx + \int_{(0,\infty)}\ \delta_n(x)e^{-2\pi ixy}\ dx = 2\int_{(0,\infty)}\ \delta_n(x)\cos(2\pi xy)\ dx.
\end{split}
\end{equation*}

\noindent
due to the fact that every $\delta_n$ is symmetric. Since 
$$1\geqslant \cos(2\pi xy) \geqslant \cos\left(2\pi\cdot \frac{y}{n}\right)$$

\noindent
for $x \in \left[0,\frac{1}{n}\right)$ and $n$ large enough then
$$1 = 2\int_{(0,\infty)}\ \delta_n(x)\ dx \geqslant \Fourier(\delta_n)(y) \geqslant 2\cos\left(2\pi\cdot \frac{y}{n}\right)\int_{(0,\infty)}\ \delta_n(x)\ dx = \cos\left(2\pi\cdot \frac{y}{n}\right).$$

\noindent
Taking the limit $n\longrightarrow\infty$ we conclude that $\Fourier(\delta_n)(y)\longrightarrow 1$ for every $y\in\reals.$

\begin{thm}
Let $T : L^1(\reals) \longrightarrow C^b(\reals)$ be an integral transform with kernel $K \in C^b(\reals\times\reals),$ which is differentiable with respect to the first variable. If $T$ satisfies 
\begin{itemize}
	\item the time differentiation property \eqref{differentiationproperty}, and
	\item the Dirac delta property \eqref{Diracdeltapropertyonreals}
\end{itemize}

\noindent
then $T$ is the Fourier transform.
\label{characterizationofFouriertransform}
\end{thm}
\begin{proof}
For every $f\in C_c^1(\reals)$ and $y\in\reals$ we have (due to integration by parts) the equality
\begin{gather}
T(f')(y) = \int_{\reals}\ K(x,y)f'(x)\ dx = -\int_{\reals}\ \frac{\partial}{\partial x}K(x,y)f(x)\ dx.
\label{integrationbyparts}
\end{gather}

\noindent
Using the time differentiation property we write
$$2\pi iy \int_{\reals}\ K(x,y)f(x)\ dx = -\int_{\reals}\ \frac{\partial}{\partial x}K(x,y)f(x)\ dx$$

\noindent
or, equivalently 
$$0 = \int_{\reals}\ \left(\frac{\partial}{\partial x}K(x,y) + 2\pi iy K(x,y)\right) f(x)\ dx$$

\noindent
for every $f\in C_c^1(\reals)$ and $y\in \reals.$ By the fundamental theorem of calculus of variations, we conclude that 
$$0 = \frac{\partial}{\partial x}K(x,y) + 2\pi iy K(x,y) = \frac{\partial}{\partial x}\left(e^{2\pi iyx}K(x,y)\right)$$

\noindent
for every $x,y\in\reals.$ This implies that 
$$K(x,y) = g(y)e^{-2\pi iyx}$$

\noindent
where $g\in C^b(\reals).$

At this stage we know that 
$$T(f)(y) = g(y)\Fourier(f)(y)$$

\noindent
for every $f\in C_c^1(\reals)$ and $y\in\reals.$ We fix $y\in\reals$ and pick any Dirac delta sequence $(\delta_n)$ for the operator $T$. We have
$$1 = \lim_{n\rightarrow\infty}\ T(\delta_n)(y) = \lim_{n\rightarrow\infty}\ g(y)\Fourier(\delta_n)(y) = g(y),$$

\noindent
which proves that $T = \Fourier$ on $C^1_c(\reals).$ However, since $C^1_c(\reals)$ is dense in $L^1(\reals)$, the two linear and bounded transforms ($T$ and $\Fourier$) must coincide on the whole space $L^1(\reals).$
\end{proof}

Motivated by the theorem above, we try to adapt the proof to characterize the Fourier transform on $\integers,$ i.e., the discrete-time Fourier transform $\Fourier : \ell^1(\integers) \longrightarrow C(S^1)$ given by
$$\Fourier(f)(y) := \sum_{n\in\integers}\ f(n)e^{-2\pi iny}.$$

\noindent
$S^1$ usually stands for the unit circle group, but in the context of the discrete-time Fourier transform, it is convenient (or customary) to understand it as an interval $[0,1]$ with addition $\bmod\ 1.$ 

Functions (sequences) in $\ell^1(\integers)$ are not differentiable per se, but we may imitate differentiability with forward and backward difference operators: 
\begin{equation*}
\begin{split}
\Delta^+f(n) &:= f(n+1) - f(n),\\
\Delta^-f(n) &:= f(n) - f(n-1).
\end{split}
\end{equation*}

\noindent
This allows for the following characterization:

\begin{thm}
If a linear and bounded operator $T : \ell^1(\integers) \longrightarrow C(S^1)$ satisfies 
\begin{itemize}
	\item the time difference property, i.e., for every $f\in \ell^1(\integers)$ and $y \in S^1$ we have
	\begin{gather}
	T(\Delta^+f)(y) = (e^{2\pi iy} - 1)T(f)(y),
	\label{differenceproperty}
	\end{gather}
	
	\item $T(\mathds{1}_0)(y) = 1$ for every $y\in S^1$ ($\mathds{1}_0$ stands for an indicator function/sequence of the singleton $\{0\}$),
\end{itemize}

\noindent
then it is the discrete-time Fourier transform.
\label{characterizationofFouriertransformZ}
\end{thm}
\begin{proof}
Since both $\integers$ and $S^1$ are $\sigma-$finite spaces, we use $L^1 - L^{\infty}$ duality to establish the existence of a kernel $K \in L^{\infty}(\integers\times S^1)$ such that 
$$T(f)(y) = \sum_{n\in\integers}\ K(n,y)f(n)$$

\noindent
for every $y\in S^1$ (see Theorem 1.3 in \cite{ArendtBukhvalov}). Next, for every $f\in \ell^1(\integers)$ and $y\in S^1$ we use summation by parts to get
\begin{equation*}
\begin{split}
T(\Delta^+f)(y) &= \lim_{n\rightarrow \infty}\ \sum_{n=-N}^N\ K(n,y)(f(n+1) - f(n)) \\
&= \lim_{N\rightarrow \infty}\ \bigg(K(N,y)f(N+1) - K(-N,y)f(-N)\bigg) \\
&- \lim_{N\rightarrow \infty}\ \sum_{n=-N}^{N-1}\ f(n+1) (K(n+1,y) - K(n,y)) \\
&= - \sum_{n\in\integers}\ f(n+1)\Delta^+K(n,y) = - \sum_{n\in\integers}\ f(n)\Delta^-K(n,y).
\end{split}
\end{equation*}

\noindent
By \eqref{differenceproperty} we have
$$(e^{2\pi iy} - 1)T(f)(y) = - \sum_{n\in\integers}\ f(n)\Delta^-K(n,y)$$

\noindent
for every $f\in \ell^1(\integers)$ and $y\in S^1$. A simple rearrangement of terms yields
$$\sum_{n\in\integers} \bigg(e^{2\pi iy}K(n,y) - K(n-1,y)\bigg)f(n) = 0,$$

\noindent
which implies 
\begin{gather}
e^{2\pi iy}K(n,y) = K(n-1,y)
\label{iteratekappa}
\end{gather}

\noindent
for every $n\in\integers$ and $y\in S^1.$ Finally, since 
$$K(0,y) = T(\mathds{1}_0)(y) = 1$$

\noindent
we may iterate \eqref{iteratekappa} to obtain $K(n,y) = e^{-2\pi iny}.$ This concludes the proof.
\end{proof}

\section{Shifts on locally compact abelian groups}
\label{Section:LCA}

Previous section focused on characterizing the Fourier transform on $\reals$ and $\integers.$ In the current section we establish a ``broader'' characterization which applies to all locally compact abelian groups, of which $\reals$ and $\integers$ are just particular examples. Naturally, we need to shift our perspective and focus on other properties of the Fourier transform, since differential or difference operators are meaningless for LCA groups in general. Thus, we bring our attention to the time shift property, which says that a time shift of signal $f$ by $x$ corresponds to frequency spectrum being modified by a linear phase shift $e^{-2\pi ixy}.$ More formally:
$$\Fourier(L_xf)(y) = e^{-2\pi ixy} \Fourier(f)(y)$$

\noindent
where $L_x: L^1(\reals)\longrightarrow L^1(\reals)$ is given by $L_xf(t) := f(t-x).$ Our goal is to prove that a generalized version of this property characterizes the Fourier transform on any LCA group.

Let $G$ be a locally compact abelian group with dual group $\widehat{G}.$ We denote the set of all open, nonempty neighbourhoods of the neutral element of $G$ by $N_0$. Next, we say that a family $(\delta_U)_{U\in N_0}$ is called a Dirac delta family, if every $\delta_U$
\begin{itemize}
	\item is continuous and compactly supported, i.e., $\delta_U \in C_c(G),$
	\item is nonnegative and symmetric, i.e., $\delta_U(x) = \delta_U(-x)$ for every $x\in G,$
	\item satisfies the equality 
	$$\int_G\ \delta_U(x)\ dx = 1.$$
\end{itemize}
 
\noindent
We say that $T:L^1(G) \longrightarrow C^b(\widehat{G})$ satisfies Dirac delta property if there exists a Dirac delta family $(\delta_U)$ such that for every $\chi\in\widehat{G}$ we have 
\begin{gather}
\lim_{U \rightarrow 1}\ T(\delta_U)(\chi) = 1,
\label{DiracdeltaLCA}
\end{gather}

\noindent
or, more explicitly, 
$$\forall_{\eps > 0}\ \exists_{U \in N_0}\ \forall_{\substack{V\subset U\\ V\in N_0}}\ |T(\delta_U)(\chi) - 1| < \eps.$$

\begin{thm}
If a linear and bounded operator $T:L^1(G) \longrightarrow C^b(\widehat{G})$ satisfies  
\begin{itemize}
	\item the time shift property, i.e., for every $f\in L^1(G)$ and $x\in G,\ \chi\in \widehat{G}$ we have 
	\begin{gather}
	T(L_xf)(\chi) = \overline{\chi(x)} T(f)(\chi),
	\label{timeshiftpropertyLCA}
	\end{gather}
	
	\item the Dirac delta property,
\end{itemize}

\noindent
then $T$ is the Fourier transform.
\label{Fouriershiftcharacterization}
\end{thm}
\begin{proof}
Let $(\delta_U)$ be a Dirac delta family for $T$. For every $f\in L^1(G),\ \chi\in\widehat{G}$ and $U \in N_0$ we use Lemma 11.45 in \cite{AliprantisBorder}, p. 427 (or Proposition 7 in \cite{Dinculeanu}, p. 123) to get
\begin{equation}
\begin{split}
T(\delta_U\star f)(\chi) &= T\left(\int_G\ f(x)L_x\delta_U\ dx\right)(\chi) = \int_G\ f(x)T(L_x\delta_U)(\chi)\ dx \\
&\stackrel{\eqref{timeshiftpropertyLCA}}{=} T(\delta_U)(\chi) \int_G\ f(x) \overline{\chi(x)}\ dx = T(\delta_U)(\chi) \Fourier(f)(\chi).
\label{mainFourierequality}
\end{split}
\end{equation}

On the other hand, $\|\delta_U\star f - f\|_1 \longrightarrow 0$ by Proposition 2.42 in \cite{Folland}, p. 53 so the continuity of $T$ implies 
$$\|T(\delta_U\star f) - T(f)\|_{\infty} \longrightarrow 0.$$ 

\noindent
Consequently, we have  
$$T(\delta_U\star f)(\chi) \longrightarrow T(f)(\chi)$$ 

\noindent
for every $f\in L^1(G)$ and $\chi\in\widehat{G}.$ Finally,
$$T(f)(\chi) = \lim_{U\rightarrow 0}\ T(\delta_U\star f)(\chi) \stackrel{\eqref{mainFourierequality}}{=} \lim_{U \rightarrow 0}\ T(\delta_U)(\chi) \Fourier(f)(\chi) \stackrel{\eqref{DiracdeltaLCA}}{=} \Fourier(f)(\chi),$$

\noindent 
which concludes the proof.
\end{proof}

\section{Fourier transform on compact groups}
\label{Section:compactgroups}

We have already established novel characterizations of the Fourier transform on $\reals, \integers$ and all locally compact abelian groups in general. The current section takes a step back and reviews the existing literature concerning the Fourier transform characterization on compact groups. With that goal in mind, let us briefly summarize the historical background.

In 2010, Jaming proved that the convolution property characterizes the Fourier transform on four canonical groups: $\reals,\ S^1,\ \integers,\ \integers_n$ (see \cite{Jaming}). Within next 4 years, Lavanya and Thangavelu proved that the Heisenberg group admits a similar characterization (see \cite{LavanyaThangavelu} and \cite{LavanyaThangavelu2}). Their approach was adapted by Kumar and Sivananthan, who investigated the convolution property of the Fourier transform on arbitrary compact group (see \cite{KumarSivananthan}). However, studying the proof of their main result, it feels like Kumar and Sivananthan take an unnecessary long detour in order to reach the final destination. It is even difficult to lay down a summary of their proof, which goes through the space of all coefficient functions and  establishes invariance of carefully crafted Hilbert subspaces in order to construct unitary operators relating the operator in question to the Fourier transform. An interested Reader is encouraged (at their own risk) to consult \cite{KumarSivananthan} for meticulous details. 

Our goal is to show that there exists a more direct path than the one chosen by Kumar and Sivananthan. The proof we lay down is rooted in representation theory and circumvents the technicalities of the earlier approach by using the correspondence between representations of a compact group and its $L^1-$space. With the clarity of our exposition in mind, we take the liberty of recalling the basic concepts of representation theory. 

Let $G$ be a compact group. A map $\pi : G \longrightarrow U(H_{\pi})$ is called a unitary representation of $G$ if
\begin{itemize}
	\item $H_{\pi}$ is a nonzero Hilbert space and $U(H_{\pi})$ is a space of unitary operators on $H_{\pi},$
	\item $\pi$ is a homomorphism, i.e., $\pi(xy) = \pi(x)\pi(y)$ and $\pi(x^{-1}) = \pi(x)^{-1} = \pi(x)^*$ for every $x,y\in G,$
	\item $\pi$ is continuous in the strong operator topology, i.e., the map $x \mapsto \pi(x)u$ if continuous for any $u \in H_{\pi}.$
\end{itemize}
 
\noindent
An instructive example of a representation is the left regular representation $\pi_L: G \longrightarrow U(L^2(G))$ defined by 
$$\pi_L(x)(f)(y) := L_xf(y) = f(x^{-1}y).$$

Two representations $\pi_1,\pi_2$ are said to be unitarily equivalent if there exists a unitary operator $Q: H_{\pi_1} \longrightarrow H_{\pi_2}$ such that 
$$\pi_2(x) = Q \pi_1(x) Q^{-1}$$

\noindent
for every $x\in G.$ For instance, the right regular representation $\pi_R: G \longrightarrow U(L^2(G))$ defined by 
$$\pi_R(x)(f)(y) := R_xf(y) = f(yx)$$ 

\noindent
is unitarily equivalent to $\pi_L$. Indeed, the unitary operator $Q : L^2(G) \longrightarrow L^2(G)$ given by $Q(f)(y) := f(y^{-1})$ satisfies
$$Q\pi_L(x)Q^{-1}f(y) = Q\pi_L(x)f(y^{-1}) = Qf(x^{-1}y^{-1}) = Qf((yx)^{-1}) = f(yx) = R_xf(y) = \pi_R(x)(f)(y)$$

\noindent
for every $x,y\in G.$ Given a representation $\pi$, the set of all unitarily equivalent representations is called the unitary equivalence class and formally denoted by $[\pi].$ However, such a symbol is rather cumbersome and we usually abuse the notation by writing $\pi$ instead of $[\pi].$

A closed subspace $S$ of $H_{\pi}$ is called invariant if $\pi(x)S\subset S$ for every $x\in G.$ Obviously, the trivial space $\{0\}$ as well as the whole space $H_{\pi}$ are invariant. If it happens that these are the only invariant subspaces of $\pi$, we say that the representation is irreducible. The set of all unitary equivalence classes of irreducible unitary representations of $G$ is denoted by $\widehat{G}$.

%

For our purposes we cannot restrict ourselves only to representations on compact groups, but rather extend the theory to $L^1-$spaces. In general, a map $\rho : A \longrightarrow B(H_{\rho})$ is called a *-representation of a Banach *-algebra $A$ if
\begin{itemize}
	\item $H_{\pi}$ is a nonzero Hilbert space and $B(H_{\pi})$ is a space of linear and bounded operators on $H_{\pi},$
	\item $\rho(a+b) = \rho(a) + \rho(b),$
	\item $\rho(\lambda a) = \lambda\rho(a),$
	\item $\rho(ab) = \rho(a)\rho(b),$
	\item $\rho(a^*) = \rho(a)^*.$
\end{itemize}

\noindent
for every $a,b\in A$ and $\lambda\in\complex.$ We do not need to assume that a *-representation is continuous, since this is always the case due to Proposition 1.3.7 in \cite{Dixmier}, p. 9. We say that a *-representation $\rho$ is nondegenerate (see Proposition 9.2 and Definition 9.3 in \cite{Takesaki}, p. 36) if
\begin{itemize}
	\item for every $u\in\Hilbert_{\rho}$ there exists $a\in A$ such that $\rho(a)u\neq 0$, or equivalently
	\item the vector space $\pi(L^1(G))H_{\pi}$ is dense in $H_{\pi}$.
\end{itemize} 

At this point we are ready to describe the correspondence between representations on compact groups and $L^1-$spaces. For every unitary representation $\pi$ of $G$ we define $\Pi: L^1(G) \longrightarrow B(H_{\pi})$ with the formula
\begin{gather}
\forall_{u,v\in H_{\pi}}\ \langle \Pi(f)u|v\rangle := \int_G\ f(x) \langle \pi(x)u|v\rangle\ dx,
\label{representationofL1}
\end{gather}

\noindent
where $\langle\cdot|\cdot\rangle$ is the inner product in the Hilbert space $H_{\pi}.$ By Theorem 3.9 in \cite{Folland}, p. 73 (or Proposition 6.2.1 in \cite{DeitmarEchterhoff}, p. 131) the map $\Pi$ is a nondegenerate *-representation of $L^1(G).$ Furthermore, by Theorem 3.11 in \cite{Folland}, p. 74 (or Proposition 6.2.3 in \cite{DeitmarEchterhoff}, p. 133) every *-representation is of the form \eqref{representationofL1}, i.e., the map $\pi \mapsto \Pi$ is a bijection between unitary representations of $G$ and nondegenerate *-representation of $L^1(G)$. Since a unitary representation $\pi$ is irreducible if and only if $\Pi$ is irreducible (see Theorem 3.12(b) in \cite{Folland}, p. 75) then the map $\pi \mapsto \Pi$ determines a bijection between $\widehat{G}$ and the set of irreducible *-representations of $L^1(G)$, denoted by $\widehat{L^1(G)}$ (see Remark 6.2.4 in \cite{DeitmarEchterhoff}, p. 134).  

Let
$$\ell^{\infty}-\oplus_{\pi \in \widehat{G}} B(H_{\pi}) := \bigg\{(A_{\pi}) \in \prod_{\pi\in\widehat{G}}\ B(H_{\pi}) \ :\ \sup_{\pi\in\widehat{G}}\ \sup_{\|u\| = 1}\ \|A_{\pi}u\|_{\pi} < \infty\bigg\}$$

\noindent
be a Banach *-algebra with the norm 
$$\|(A_{\pi})\|_{\infty} := \sup_{\pi\in\widehat{G}}\ \sup_{\|u\| = 1}\ \|A_{\pi}u\|_{\pi}$$

\noindent
and involution $(A_{\pi})^* = (A_{\pi}^*).$ The Fourier transform on compact group $G$ is the map $\Fourier: L^1(G) \longrightarrow \ell^{\infty}-\oplus_{\pi \in \widehat{G}} B(H_{\pi})$ given by 
\begin{gather*}
\forall_{u,v\in H_{\pi}}\ \langle \Fourier(f)(\pi)u|v\rangle := \int_G\ f(x) \langle\pi(x)^*u|v\rangle\ dx.
\end{gather*} 

\noindent

\begin{thm}
Let $T : L^1(G) \longrightarrow \ell^{\infty}-\oplus_{\pi \in \widehat{G}} B(H_{\pi})$ be a linear and bounded operator, which is
\begin{itemize}
	\item *-preserving, i.e., $T(f)^* = T(f^*)$ for every $f\in L^1(G)$, and
	\item pointwise irreducible, i.e., for every $\pi\in\widehat{G}$ the map $f \mapsto T(f)(\pi)$ is irreducible.
\end{itemize}

\noindent
If $T$ satisfies 
\begin{itemize}
	\item the convolution property, i.e., for every $f,g \in L^1(G)$ and $\sigma\in \widehat{G}$ we have
		\begin{gather}
		T(f\star g)(\sigma) = T(f)(\sigma) T(g)(\sigma)
		\label{convolutionpropertycompact}
		\end{gather}
		
	\item the time shift property, i.e., for every $f\in L^1(G), x \in G$ and $\pi \in \widehat{G}$ we have 
		\begin{gather}
		 T(L_xf)(\pi) = T(f)(\pi)\pi(x)^*
		\label{leftshiftcompact}
		\end{gather}
\end{itemize}

\noindent
then $T$ is the Fourier transform  
\label{maintheorem}
\end{thm}
\begin{proof}
We choose $\pi \in \widehat{G}$ and consider a linear map $T_{\pi}: L^1(G)\longrightarrow B(H_{\pi})$ given by $T_{\pi}(f) := T(f)(\pi).$ Since $T$ is *-preserving, pointwise irreducible and it satisfies the convolution property then $T_{\pi}$ is an irreducible *-representation of $L^1(G).$ Consequently, there exists an irreducible unitary representation $\sigma : G\longrightarrow U(H_{\pi})$ such that 
$$\langle T_{\pi}(f)u|v\rangle = \int_G\ f(x) \langle \sigma(x)^*u|v\rangle\ dx$$

\noindent
for every $u,v \in H_{\pi}.$

Next, for every $f \in L^1(G), x\in G$ and $u,v\in H_{\pi}$ we have
\begin{equation*}
\begin{split}
\int_G\ f(y)\langle \sigma(y)^*\pi(x)^*u|v \rangle\ dy &= \langle T_{\pi}(f)\pi(x)^*u|v\rangle \stackrel{\eqref{leftshiftcompact}}{=} \langle T_{\pi}(L_xf)u|v\rangle\\ 
&= \int_G\ L_xf(y) \langle \sigma(y)^*u|v \rangle \ dy = \int_G\ f(x^{-1}y) \langle \sigma(y)^*u|v \rangle \ dy \\
&\stackrel{y\mapsto xy}{=} \int_G\ f(y) \langle \sigma(xy)^*u|v \rangle \ dy = \int_G\ f(y) \langle \sigma(y)^* \sigma(x)^*u|v \rangle \ dy.
\end{split}
\end{equation*}

\noindent
A simple rearrangement yields 
$$\int_G\ f(y) \langle \sigma(y)^* (\sigma(x)^* - \pi(x)^*)u|v \rangle = 0$$

\noindent
for every $f\in L^1(G), x \in G$ and $u,v\in H_{\pi}.$ Putting
$$f(y) := \overline{\langle \sigma(y)^* (\sigma(x)^* - \pi(x)^*)u|v \rangle}$$

\noindent
we arrive at the conclusion that 
$$\langle \sigma(y)^* (\sigma(x)^* - \pi(x)^*)u|v \rangle = 0$$

\noindent
for every $x,y \in G$ and $u,v\in H_{\pi}.$ Finally, we take $y$ to be the identity element in $G$ and $v = (\sigma(x)^* - \pi(x)^*)u$ to obtain $(\sigma(x)^* - \pi(x)^*)u = 0$ for every $x\in G$ and $u\in H_{\pi}.$ This means that $\sigma = \pi,$ which completes the proof. 
\end{proof}

\section{Beyond the Fourier transform}
\label{Section:Hankel}

This last section is meant to stimulate research regarding integral transforms and their characterizations. As the title suggests, our goal is to prove that the Fourier transform is not the only well-known operator characterized by a differential property. We hope that our next result will become a catalyst for future study of Fourier-like integral transform and their characterizations. 

To begin with, let $J_{\alpha}$ be the Bessel's function of the first kind solving the Bessel's differential equation (of order $\alpha$):
\begin{gather}
r^2 f''(r) + r f'(r) + (r^2 - \alpha^2)f(r) = 0
\label{Besselequation}
\end{gather}

\noindent
for every $r\in\reals_+.$ One form of expressing $J_{\alpha}$ is (see Chapter 9.4 in \cite{Temme}, p. 230 or Chapter VI in \cite{Watson}, p. 176):
$$\forall_{r\in\reals_+}\ J_{\alpha}(r) = \frac{1}{\pi} \int_0^{\pi}\ \cos(\alpha t - r\sin(t))\ dt - \frac{\sin(\alpha\pi)}{\pi} \int_0^{\infty}\ e^{-r\sinh(t)-\alpha t}\ dt,$$

\noindent
which implies that $J_{\alpha}$ is a bounded function. This feature distinguishes $J_{\alpha}$ from $Y_{\alpha},$ which is another solution of \eqref{Besselequation} called the Bessel's function of the second kind. $Y_{\alpha}$ is unbounded at the origin $r = 0.$ 

Next, we focus on the Hankel transform $\Hankel_{\alpha}: C_c(\reals_+)\longrightarrow C_0(\reals_+)$ of order $\alpha:$
\begin{gather}
\Hankel_{\alpha}(f)(y) := \int_{\reals_+}\ J_{\alpha}(yx)f(x)x\ dx.
\label{Hankeltransformdefinition}
\end{gather}

\noindent
Continuity of $\Hankel_{\alpha}(f)$ follows from the dominated convergence theorem, boundedness of $J_{\alpha}$ and the function $f$ having compact support. Furthermore, in order to see that $\Hankel_{\alpha}(f)$ vanishes at infinity, we make the substitution $x\mapsto \frac{x}{y}$ in \eqref{Hankeltransformdefinition} to obtain
$$\Hankel_{\alpha}(f)(y) = \int_{\reals_+}\ J_{\alpha}(x)f\left(\frac{x}{y}\right)\frac{x}{y^2}\ dx.$$ 

\noindent
Again, it suffices to apply the dominated convergence theorem to conclude that 
$$\lim_{y\rightarrow\infty}\ H_{\alpha}(f)(y) = 0$$

\noindent
for every $f\in C_c(\reals_+).$ We are now ready to prove the final result of the paper:

\begin{thm}
Let $T : C_c(\reals_+) \longrightarrow C_0(\reals_+)$ be an integral transform given by the formula
\begin{gather}
T(f)(y) := \int_{\reals_+}\ K(yx)f(x)x\ dx,
\label{Hankeltransform}
\end{gather}

\noindent
where $K\in C^2(\reals_+)$ is a bounded function. If
\begin{itemize}
	\item $T$ satisfies the Bessel's differential property (of order $\alpha>0$) 
	\begin{gather}
	\forall_{f\in C_c^2(\reals_+)}\ T\left(f'' + \frac{1}{x} f' - \frac{\alpha^2}{x^2}f\right)(y) = -y^2 T(f)(y),
	\label{Besseldifferentiationproperty}
	\end{gather}

	\noindent
	and
	\item there exists a function $f_*\in C_c^2(\reals_+)$ and $y_*\in\reals_+$ such that 
	\begin{gather}
	T(f_*)(y_*) = \Hankel_{\alpha}(f_*)(y_*) \neq 0,
	\label{Besselsecondassumption}
	\end{gather}
\end{itemize}

\noindent
then $T$ is the Hankel transform $\Hankel_{\alpha}$.
\label{Hankelcharacterization}
\end{thm}
\begin{proof}
In order to facilitate the computations throughout the proof, let us begin with a pair of auxiliary integrations by parts. For every $f\in C_c^2(\reals_+)$ and $y\in\reals_+$ we have
\begin{align}
\int_{\reals_+}\ xK(yx)f''(x)\ dx &= - \int_{\reals_+}\ K(yx)f'(x)\ dx - \int_{\reals_+}\ yxK'(yx) f'(x)\ dx, \label{align:1}\\
\int_{\reals_+}\ yx K'(yx) f'(x)\ dx &= - \int_{\reals_+}\ yK'(yx) f(x)\ dx  - \int_{\reals_+}\ y^2xK''(yx) f(x)\ dx. \label{align:2}
\end{align}

Due to Bessel's differential property, for every $f\in C^2_c(\reals_+)$ and $y\in\reals_+$ we have
\begin{equation*}
\begin{split}
-y^2 T(f)(y) &\stackrel{\eqref{Besseldifferentiationproperty}}{=} T\left(f'' + \frac{1}{x} f' - \frac{\alpha^2}{x^2}f\right)(y) = \int_{\reals_+}\ K(yx)\left(f''(x) + \frac{1}{x}f'(x) - \frac{\alpha^2}{x^2}f(x)\right)x\ dx \\
&= \int_{\reals_+}\ xK(yx)f''(x)\ dx + \int_{\reals_+}\ K(yx)f'(x)\ dx - \int_{\reals_+}\ K(yx)\frac{\alpha^2}{x}f(x)\ dx \\
&\stackrel{\eqref{align:1}}{=} - \int_{\reals_+}\ yxK'(yx) f'(x)\ dx - \int_{\reals_+}\ K(yx)\frac{\alpha^2}{x}f(x)\ dx \\
&\stackrel{\eqref{align:2}}{=} \int_{\reals_+}\ \left(y^2xK''(yx) + yK'(yx) - \frac{\alpha^2}{x}K(yx)\right) f(x)\ dx,
\end{split}
\label{mainHankelcomputation}
\end{equation*}

\noindent
which can be rewritten as 
$$0 = \int_{\reals_+}\ \bigg((yx)^2K''(yx) + yxK'(yx) - \big((yx)^2 - \alpha^2\big) K(yx)\bigg) \frac{f(x)}{x}\ dx.$$

\noindent
By the fundamental theorem of calculus of variations we have 
$$r^2K''(r) + rK'(r) - (r^2 - \alpha^2) K(r) = 0$$

\noindent
for every $r\in\reals_+.$ This implies that $K$ is a linear combination of Bessel's functions $J_{\alpha}$ and $Y_{\alpha}$ of first and second kind, respectively, so
$$K(r) = C_1 J_{\alpha}(r) + C_2 Y_{\alpha}(r)$$

\noindent
for some constants $C_1, C_2 \in\reals.$ Since $Y_{\alpha}$ is unbounded near the origin (and $J_{\alpha}$ is bounded), it must be the case that $C_2 = 0$, as the kernel $K$ is (by assumption) bounded. This means that $T = C_1\Hankel_{\alpha}.$

Finally, we have
$$\Hankel_{\alpha}(f_*)(y_*) \stackrel{\eqref{Besselsecondassumption}}{=} T(f_*)(y_*) = C_1\Hankel_{\alpha}(f_*)(y_*),$$

\noindent
which leads to the conclusion that $C_1 = 1$ and completes the proof. 
\end{proof}


\begin{thebibliography}{9}
	\bibitem{AliprantisBorder}
		Aliprantis C. D., Border K. C. : \textit{Infinite Dimensional Analysis. A Hitchhiker's Guide}, Springer-Verlag, Berlin, 2006
	\bibitem{ArendtBukhvalov}
		Arendt W., Bukhvalov A. V. : \textit{Integral representations of resolvents and semigroups}, Forum Math., Vol. 6, p. 111-135 (1994)
	\bibitem{DeitmarEchterhoff}
	  Deitmar A., Echterhoff S. : \textit{Principles of Harmonic Analysis}, Springer, New York, 2009
	\bibitem{Dinculeanu}
		Dinculeanu N. : \textit{Vector measures}, Pergamon Press, Berlin, 1967
	\bibitem{Dixmier}
		Dixmier J. : \textit{C*-Algebras}, North-Holland Publishing Company, New York, 1977.
	\bibitem{Folland}
		Folland G. : \textit{A Course in Abstract Harmonic Analysis}, CRC Press, London, 1995
	\bibitem{Jaming}
		Jaming P. : \textit{A characterization of Fourier transforms}, Coll. Math., Vol. 118 (2), p. 569-580 (2010)
	\bibitem{KumarSivananthan}
		Kumar N. S.,  Sivananthan S. : \textit{Characterisation of the Fourier transform on compact groups}, Bull. Aust. Math. Soc., Vol. 93, p. 467-472 (2016)
	\bibitem{LavanyaThangavelu}
		Lavanya R. L., Thangavelu S. : \textit{A characterization of the Fourier transform on the Heisenberg group}, Ann. Funct. Anal., Vol. 3(1), 109-120 (2012)
	\bibitem{LavanyaThangavelu2}
		Lavanya R. L., Thangavelu S. : \textit{Revisiting the Fourier transform on the Heisenberg group}, Publ. Mat., Vol. 58, p. 47-63 (2014)
	\bibitem{Takesaki}	
		Takesaki M. : \textit{Theory of Operator Algebras I}, Springer, Berlin, 2002.
	\bibitem{Temme}
		Temme N. M. : \textit{Special Functions. An Introduction to the Classical Functions of Mathematical Physics}, John Wiley and Sons, New York, 1996
	\bibitem{Watson}
		Watson G. N. : \textit{A treatise on the theory of Bessel functions}, Cambridge University Press, Cambridge, 1944
\end{thebibliography}
\end{document}